\documentclass[final]{siamltex}

\usepackage{color}
\usepackage{a4,amsmath,amssymb,amscd,latexsym} 
\usepackage{bbm} 
\usepackage{epsfig,color} 
\usepackage{amsfonts}
\usepackage{mathrsfs}
\usepackage{overpic}

\usepackage{paralist}

\usepackage{graphicx}

\usepackage{trfsigns}
\usepackage{algpseudocode} 

\usepackage{stmaryrd}
\usepackage{caption}
\usepackage{subfigure}

\newcommand{\Gi}{\mathrm{\Gamma_{int}}}
\newcommand{\Gb}{\mathrm{\Gamma_{bottom}}}
\newcommand{\Gl}{\mathrm{\Gamma_{left}}}
\newcommand{\Gr}{\mathrm{\Gamma_{right}}}
\newcommand{\Gt}{\mathrm{\Gamma_{top}}}
\newcommand{\Go}{\mathrm{\Gamma_{out}}}

\def\norm#1{\|#1\|} 

\def\nd#1{\frac{\partial #1}{\partial n}} 
\def\td#1{\frac{\partial #1}{\partial t}} 

\def\intdxdt#1{\int_{0}^{T} \int_{\Omega} #1 \hspace{.5mm}dx\hspace{.5mm}dt}
\def\intdtdx#1{\int_{\Omega} \int_{0}^{T} #1 \hspace{.5mm}dt\hspace{.5mm}dx}
\def\intdx#1{\int_{\Omega} #1 \hspace{1mm}dx}
\def\intdsidt#1{\int_{0}^{T}\int_{\Gi} #1 \hspace{.5mm}ds\hspace{.5mm}dt}
\def\intdtdsi#1{\int_{\Gi} \int_{0}^{T}#1 \hspace{.5mm}dt\hspace{.5mm}ds}

\def\intdsodt#1{\int_{0}^{T} \int_{\Go} #1 \hspace{.5mm}ds\hspace{.5mm}dt}
\def\intdtdso#1{\int_{\Go} \int_{0}^{T} #1 \hspace{.5mm}dt\hspace{.5mm}ds}
\def\intdt#1{\int_{0}^{T} #1 \hspace{.5mm}dt}

\def\grad{\mbox{grad}}
\newcommand{\LL}{{\mathscr{L}}}

\def\bei#1{\vrule width 0.4pt height 14pt depth 9pt
           \lower 8pt \hbox{$ _{\hbox{} #1}$}\!\!\!}

\def\exp{\mathrm{exp}}

\newcommand{\N}{{\mathbbm{N}}} 
\newcommand{\R}{{\mathbbm{R}}} 

\makeatletter
\def\moverlay{\mathpalette\mov@rlay}
\def\mov@rlay#1#2{\leavevmode\vtop{%
   \baselineskip\z@skip \lineskiplimit-\maxdimen
   \ialign{\hfil$\m@th#1##$\hfil\cr#2\crcr}}}
\newcommand{\charfusion}[3][\mathord]{
    #1{\ifx#1\mathop\vphantom{#2}\fi
        \mathpalette\mov@rlay{#2\cr#3}
      }
    \ifx#1\mathop\expandafter\displaylimits\fi}
\makeatother

\newcommand{\cupdot}{\charfusion[\mathbin]{\cup}{\cdot}}

\def\scp#1{\left\langle #1\right\rangle}
\def\norm#1{\left\| #1\right\|}

\def\CROP#1{}

\newtheorem{remark}{Remark}

\title{Structured inverse modeling in parabolic diffusion problems
}

\author{Volker H.~Schulz, Martin Siebenborn and Kathrin Welker\thanks{University of Trier, Department of Mathematics, 54296 Trier, Germany ({\tt volker.schulz@uni-trier.de, siebenborn@uni-trier.de, welker@uni-trier.de}).}
        }

\newenvironment{rev} {\color{black}} {\color{black}}

\begin{document}

\maketitle

\begin{abstract}
Often, the unknown diffusivity in diffusive processes is structured by piecewise constant patches. This paper is devoted to efficient methods for the determination of such structured diffusion parameters by exploiting shape calculus. A novel shape gradient is derived in parabolic processes. Furthermore quasi-Newton techniques are used in order to accelerate shape gradient based iterations in shape space. Numerical investigations support the theoretical results.
\end{abstract}



\pagestyle{myheadings}
\thispagestyle{plain}

\section{Introduction}
Inverse modeling in diffusive processes is one of the major themes in the field of inverse problems. 
\begin{rev} 
Inverse problems were already tackled for example in \cite{ColtonKress,HettlichRundell}.\end{rev} 
Often, a distributed diffusivity parameter is to be estimated from observations of the diffused state, as in \cite{McLaughlin1996,SchulzVxX1999s,SchulzVxH1997c}. In many cases, however, the rough overall structure of the parameter distribution is known, but the details are missing. In the present paper, we assume that the distributed diffusion parameter to be estimated is piecewise constant in subdomains with smooth boundaries. The detailed shape of the subdomains is to be estimated. Thus, we elaborate on a very similar setting as in \cite{harbrecht-2014}. The difference is that in \cite{harbrecht-2014} the source term is assumed being piecewise constant, whereas here the diffusion parameter is assumed piecewise constant. Furthermore, a novel quasi-Newton approach in shape space is presented and  convergence properties are observed, which are superlinear as long as the increments are larger than the discretization error. 
\begin{rev}
Newton-type methods have been used in shape optimization since many years, e.g. \cite{Epp-Har-2005,NovruziRoche}.\end{rev}
Quasi-Newton methods on general manifolds have already been discussed in \cite{Absil-book-2008,GabayDxX1982b,Ring-Wirth-2012}. Here, we specify them for the particular case of shape manifolds. From a different standpoint, the discussion in this paper can be viewed as a  generalization of the elliptic structured inverse modeling in the publications  \cite{ItoKunisch,Paganini} to the parabolic case.
\begin{rev}
The methodology and algorithm derived in this paper applies for example to the problem of inversely determining cell shapes in the human skin as investigated in \cite{neagel20015scalable}.\end{rev}

The paper is organized in the following way. 
In section 2, we derive the shape derivative for the parabolic inverse problem. Section 3 presents a limited memory BFGS quasi-Newton technique in shape space and discusses the theoretical background from optimization on Riemannian manifolds. Finally, section 4 discusses numerical results for the inverse problem of finding the interfaces of two subdomains. 

\section{Interface problem formulation and derivation of the shape derivative}\label{sec3}
We first set up notation and terminology. Then we formulate the parabolic interface problem which is motivated by electrical impedance tomography. In the third part of this section we deduce the shape derivative which is achieved by an application of the theorem of Correa and Seger \cite[theorem 2.1]{CorreaSeger} and a generalization of the approach in \cite{Paganini} for parabolic problems.
\subsection{Notations and definitions}\label{notations_definitions}
Let $d\in\N$ and $\tau>0$. We will denote by $\Omega\subset \R^d$ a bounded domain with Lipschitz boundary $\Gamma:=\partial\Omega$ and by $J$ a real-valued functional depending on it.
Moreover, let $\{F_t\}_{t\in[0,\tau]}$ be a family of bijective mappings $F_t\colon \Omega\to\R^d$ such that $F_0=id$.
This family transforms the domain $\Omega$ into new perturbed domains $\Omega_t:=F_t(\Omega)=\{F_t(x)\colon x\in \Omega\}$ with $\Omega_0=\Omega$ and the boundary $\Gamma$ into new perturbed boundaries $\Gamma_t:=F_t(\Gamma)=\{F_t(x)\colon x\in \Gamma\}$ with $\Gamma_0=\Gamma$. If you consider the domain $\Omega$ as a collection of material particles which are changing their position in the time-interval $[0,\tau]$, then the family $\{F_t\}_{t\in[0,\tau]}$ describes the motion of each particle, i.e., at the time $t\in [0,\tau]$ a material particle $x\in\Omega$ has the new position $x_t:=F_t(x)\in\Omega_t$ with $x_0=x$.
The motion of each such particle $x$ could be described by the \emph{velocity method}, i.e., as the flow $F_t(x):=\xi(t,x)$ determined by the initial value problem
\begin{equation}
\begin{split}
\frac{d\xi(t,x)}{dt}&=V(\xi(t,x))\\
\xi(0,x)&=x
\end{split}
\end{equation}
or by the \emph{perturbation of identity} which is defined by $F_t(x):=x+tV(x)$ where $V$ denotes a sufficiently smooth vector field. We will use the perturbation of identity throughout the paper. The \emph{Eulerian derivative} of $J$ at $\Omega$ in direction $V$ is defined by
\begin{equation}
\label{eulerian}
DJ(\Omega)[V]:= \lim\limits_{t\to 0^+}\frac{J(\Omega_t)-J(\Omega)}{t}.
\end{equation}
The expression $DJ(\Omega)[V]$ is called the \emph{shape derivative} of $J$ at $\Omega$ in direction $V$ and $J$ \emph{shape differentiable} at $\Omega$ if for all directions $V$ the Eulerian derivative (\ref{eulerian}) exists and the mapping $V\mapsto DJ(\Omega)[V]$ is linear and continuous. The \emph{material derivative} of a generic function $p\colon \Omega_t\to \R$ at $x\in\Omega$ with respect to the deformation $F_t$ is given by
\begin{equation}
\label{material}
D_mp(x):=\lim\limits_{t\to 0^+}\frac{\left(p\circ F_t\right)(x)-p(x)}{t}=\frac{d^+}{dt}\left(p\circ F_t\right)(x)\,\rule[-2.5mm]{.1mm}{6mm}_{\hspace{1mm}t=0}
\end{equation}
and its \emph{shape derivative} with respect to the vector field $V$ by
\begin{equation}
\label{shape_der}
Dp[V]:=D_mp-V^T\nabla p.
\end{equation}
In the following, we will also use the symbol $\dot{p}$ to denote the material derivative of $p$.
Let $p,q\colon \Omega_t\to\R$ be two generic functions and $D_m$ the material derivative with respect to $F_t=id+tV$. The following rules for the material will be needed in subsection \ref{der_shape_der}. For the material derivative the product rule holds, i.e.,
\begin{equation}
\label{product_rule}
D_m(p\hspace{.7mm}q)=D_mp\hspace{.7mm}q+p\hspace{.7mm}D_mq.
\end{equation}
While the shape derivative commutes with the gradient, the material derivative does not, but the following equality was proved in \cite{Berggren}
\begin{equation}
\label{material_grad}
D_m\nabla p=\nabla D_mp-\nabla V^T\nabla p.
\end{equation}
Combining (\ref{product_rule}) with (\ref{material_grad}) yields
\begin{equation}
\label{material_grad_grad}
D_m\left(\nabla q^T\nabla p\right)=\nabla D_mp^T\nabla q-\nabla q^T\left(\nabla V+\nabla V^T\right)\nabla p+\nabla p^T\nabla D_mq.
\end{equation}
Moreover, in subsection \ref{der_shape_der} we need the following rule for differentiating domain integrals
\begin{equation}
\label{der_domain_int}
\frac{d^+}{dt}\left(\int_{\Omega_t}f(t)\right)\,\rule[-4mm]{.1mm}{9mm}_{\hspace{1mm}t=0}=\int_{\Omega}\left(D_mf+\mathrm{div}(V)f\right)
\end{equation}
which was proved in \cite[lemma 3.3]{HaslingerMakinen}.

\subsection{Interface problem formulation}
In the previous subsection we denoted by $\Omega$ a bounded domain of $\R^d$ with Lipschitz boundary $\partial\Omega$. Now, let this domain $\Omega$ be an open subset of $\R^2$ and split into the two disjoint subdomains $\Omega_1,\Omega_2\subset\Omega$ such that $\Omega_1\cupdot\Gi\cupdot\Omega_2=\Omega$, $\Gb\cupdot\Gl\cupdot\Gr\cupdot\Gt=\partial\Omega\hspace{1mm}(=:\Go)$ and $\partial\Omega_1\cap\partial\Omega_2=\Gi$ where the interior boundary $\Gi$ is assumed to be smooth and variable and the outer boundary $\Go$ Lipschitz and fixed. An example of such a domain is illustrated in figure \ref{fig_Omega}.
\begin{figure}[h]
\vspace*{.6cm}
\begin{center}
   \begin{overpic}[width=.5\textwidth]{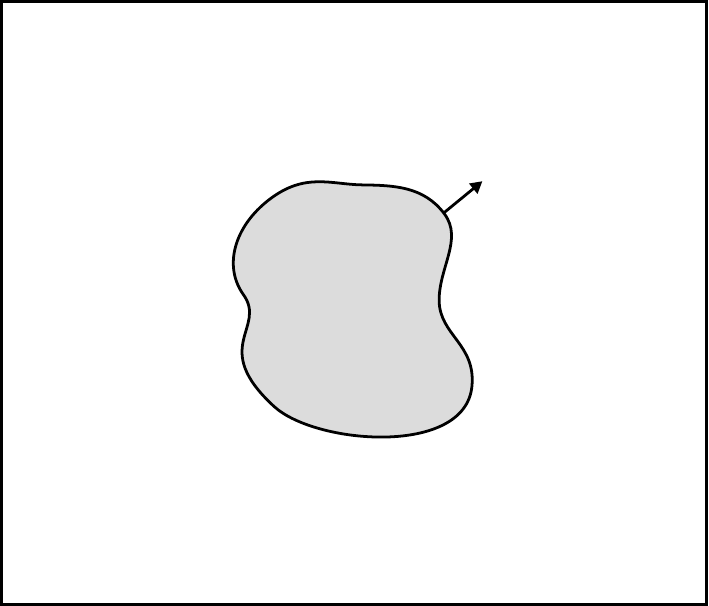}
   \put(42,40){$\Omega_2$}
   \put(20,13){$\Omega_1$}
   \put(44,88){$\Gt $}
   \put(43,61.5){$\Gi $}
   \put(-11,42){$\Gl $}
   \put(101,42){$\Gr $}
   \put(44,-6){$\Gb $}
   \put(65.5,54){$n$}
   \end{overpic}
   \end{center}
   \vspace*{.4cm}
   \caption{Example of a domain $\Omega=\Omega_1\cupdot\Gi\cupdot\Omega_2$ where $\Go:=\partial\Omega=\Gb\cupdot\Gl\cupdot\Gr\cupdot\Gt$ and $n$ denotes the unit outer normal to $\Omega_2$ at $\Gi$}
   \label{fig_Omega}
\end{figure}

The parabolic PDE constrained shape optimization problem is given in strong form by
\begin{align}\label{oc1}
\min \hspace{0,1cm} J(\Omega):=\intdxdt{&(y-\bar{y})^2}+\mu\int_{\Gi}1\hspace{.5mm}ds\\
\label{oc2}
\mbox{s.t. } \td{y}- \mathrm{div}(k\nabla y)&=f\quad \text{in }\Omega\times(0,T]
\\
\label{oc3}
\hspace{20mm}y&=1\quad \text{on }\Gt\times(0,T]
\\
\label{oc4}
\nd{y}&=0\quad \text{on }(\Gb\cup\Gl\cup\Gr)\times(0,T]
\\
\label{oc5}
y&=y_0\quad\text{in }\Omega\times\{0\}
\end{align}
where
\begin{equation*}
k\equiv\begin{cases}
k_1 = \mathrm{const.}\quad\text{ in }\Omega_1\times(0,T]\\
k_2 = \mathrm{const.} \quad\text{ in }\Omega_2\times(0,T]
\end{cases}
\end{equation*}
and $n$ denotes the unit outer normal to $\Omega_2$ at $\Gi$. Of course, the formulation (\ref{oc2}) of the differential equation is to be understood only formally because of the jumping coefficient $k$. We observe that the unit outer normal to $\Omega_1$ is equal to $-n$, which enables us to use only one normal $n$ for the subsequent discussions. Furthermore, we have interface conditions at the interface $\Gi$. We formulate explicitly the continuity of the state and of the flux at the boundary as
\begin{equation}
\label{oc6}
\left\llbracket y\right\rrbracket =0\, ,\quad \left\llbracket k\nd{y}\right\rrbracket =0\quad\text{on }\Gi\times(0,T]
\end{equation}
where the jump symbol $\llbracket\cdot\rrbracket$ denotes the discontinuity across the interface $\Gi$ and is defined by $\llbracket v \rrbracket := v_1-v_2$ 
where $v_1:=v\,\rule[-2mm]{.1mm}{4mm}_{\hspace{.5mm}\Omega_1}$ and $v_2:=v\,\rule[-2mm]{.1mm}{4mm}_{\hspace{.5mm}\Omega_2}$. The perimeter regularization with $\mu>0$ in the objective (\ref{oc1}) is a frequently used \begin{rev} in this kind of problems. In \cite{Sturm2013} a weaker but more complicated regularization is instrumental in order to show existence of solutions. \end{rev}
 
\begin{rev}
It is often important to identify functions $\Omega \times [0,T]\to \R$ with maps from $[0,T]$ into a Banach space. In doing so, we now use the space $L^2\left(0,T;H^1(\Omega)\right)$ which consists of $L^2$-integrable functions $u\colon [0,T]\to H^1(\Omega)$ such that $u(t)\in H^1(\Omega)$ for all $t\in[0,T]$. Moreover, as in \cite{Troeltzsch} we now use a weak time derivative $u_t$ for $u\in L^2\left(0,T;H^1(\Omega)\right)$ by the following condition
\begin{equation}
\int_{0}^{T} \phi(t) u_t(t)dt =-\int_{0}^{T}  \phi'(t) u(t) dt, \,\ \forall\phi\in C_0^\infty(0,T)
\end{equation}
where $\phi'=\frac{d\phi}{dt}$.
In the following, we assume that $u\in L^2\left(0,T;H^1(\Omega)\right)$ has a weak time derivative $u_t\in L^2\left(0,T;H^{-1}(\Omega)\right)$ where $H^{-1}(\Omega)$ denotes the dual space of $H^1(\Omega)$. 
\begin{remark}
Let $H$ be a Hilbert space. A weak time derivative of a function $u\in L^2\left(0,T;H\right)$ can be defined in the same space $L^2\left(0,T;H\right)$. However, in our application, we use the Gelfand triple $H^1(\Omega)\hookrightarrow L^2(\Omega)\hookrightarrow H^{-1}(\Omega)$ and it turns out to be more appropriate to define the weak time derivative in the space $L^2\left(0,T;H^{-1}(\Omega)\right)$ instead of $L^2\left(0,T;H^1(\Omega)\right)$ (cf. \cite{GrossReusken,Troeltzsch}).
\end{remark}
\end{rev} 

In our setting, the boundary value problem (\ref{oc2}-\ref{oc6}) is written in weak form as
\begin{equation}
\label{wf}
a(y,p)=b(p,p_1,p_2)\, , \ \forall p,p_1,p_2\in \begin{rev}
W\left(0,T;H^1(\Omega)\right)
\end{rev}
\end{equation}
\begin{rev}
where the space $W\left(0,T;H^1(\Omega)\right)$ is defined by
\begin{equation}
W\left(0,T;H^1(\Omega)\right)=\{u\in L^2(0,T;H^1(\Omega))\colon u_t\in L^2\left(0,T;H^{-1}(\Omega)\right)\text{ exists}\}.
\end{equation}
For properties of the space $W\left(0,T;H^1(\Omega)\right)$ we refer the reader to the literature, e.g. \cite{GrossReusken,Troeltzsch}.
\end{rev}
The bilinear form $a(y,p)$ in (\ref{wf}) is given by
\begin{align}
a(y,p)\hspace{.5mm} &:=\intdx{y(T,x)\hspace{.5mm}p(T,x)}-\intdx{y_0\hspace{.5mm}p(0,x)}-\intdxdt{\td{p}y}\nonumber\\
& \hspace*{.5cm}+ \intdxdt{k\nabla y^T\nabla p}-\intdsidt{\left\llbracket k\nd{y}p\right\rrbracket}\nonumber\\
&\hspace*{.5cm}-\intdsodt{k_1\nd{y}p}\label{wfa2}
\end{align}
and the linear form $b(p,p_1,p_2)$ in (\ref{wf}) by
\begin{equation}
b(p,p_1,p_2)\hspace{.5mm}:=b_1(p)+b_2(p_1,p_2)
\label{wfb}
\end{equation}
where
\begin{align}
b_1(p)&:=\intdxdt{fp}\label{b1} \\
b_2(p_1,p_2)&:=\intdt{\int_{\Gt}p_1(y-1)\hspace{.5mm}ds}+\intdt{\int_{\Go\setminus\Gt}p_2\nd{y}\hspace{.5mm}ds}.\label{b2}
\end{align}
\begin{rev} We assume for the obervation $\bar{y}\in W\left(0,T;H^1(\Omega)\right)$, which guarantees also ${y},p\in W\left(0,T;H^1(\Omega)\right)$.\end{rev}
The Lagrangian of (\ref{oc1}-\ref{oc6}) is defined as
\begin{equation}
\LL(\Omega,y,p):=J(\Omega)+a(y,p)-b(p,p_1,p_2)
\label{lagrangian}
\end{equation}
where $J(\Omega)$ is defined in (\ref{oc1}), $a(y,p)$ in (\ref{wfa2}) and $b(p,p_1,p_2)$ in (\ref{wfb}-\ref{b2}).
\begin{remark}
Integration by parts on the integral $\intdxdt{\td{y}p}$ yields
\begin{align*}
&\intdxdt{\td{y}p}\\
&=\intdx{y(T,x)\hspace{.5mm}p(T,x)}-\intdx{y_0\hspace{.5mm}p(0,x)}-\intdxdt{\td{p}y}
\end{align*}
\end{remark}
\begin{remark}
\label{wfa1yp}
Note that we have to consider
\begin{align}
a(y,p)\hspace{.5mm} &:=\intdx{y(T,x)\hspace{.5mm}p(T,x)}-\intdx{y_0\hspace{.5mm}p(0,x)}-\intdxdt{\td{p}y}\nonumber\\
& \hspace*{.5cm}-\intdxdt{\mathrm{div}(k\nabla p) y}+ \intdsidt{\left\llbracket k\left(\nd{p}y-\nd{y}p\right)\right\rrbracket}\nonumber\\
&\hspace*{.5cm}+\intdsodt{k_1\left(\nd{p}y-\nd{y}p\right)}\label{wfa1y}
\end{align}
or respectively
\begin{equation}
a(y,p)\hspace{.5mm} :=\intdxdt{\td{y}p}- \intdxdt{\mathrm{div}(k\nabla y) p}\label{wfa1p}
\end{equation}
instead of (\ref{wfa2}) in order to derive the bilinear form $a(y,p)$ or the Lagrangian $\LL(\Omega,y,p)$ in terms of $y$ or respectively $p$.
\end{remark}

\subsection{Derivation of the shape derivative}
\label{der_shape_der}

In this subsection we first consider the objective (\ref{oc1}) without the perimeter regularization. Then the shape derivative can be expressed as an integral over the domain $\Omega$, as well as an integral over the interface $\Gi$. By the Hadamard structure theorem \cite[theorem 2.27]{SokoZol} only the normal part of a vector field V on the interface has an impact on the value of the shape derivative $D\LL(\Omega,y,p)[V]$ or $DJ(\Omega)[V]$. In this subsection we first deduce the domain integral by an application of the theorem of Correa and Seger \cite[theorem 2.1]{CorreaSeger}. Then we convert it in an interface integral by means of integration by parts on $\Gi$.

\begin{remark}
The shape derivative in an open domain will only depend on the normal component of a vector field on the boundary, if the boundary is smooth enough. One should note that this is no longer true, if the boundary is only piecewise smooth.
\end{remark}

A saddle point $(y,p)\in \begin{rev}W\left(0,T;H^1(\Omega)\right)\times W\left(0,T;H^1(\Omega)\right)\end{rev}$ of the Lagrangian (\ref{lagrangian}) is given by
\begin{align}
\frac{\partial \LL(\Omega,y,p)}{\partial y}=\frac{\partial \LL(\Omega,y,p)}{\partial p}=0\label{saddlepointcond}
\end{align}
which leads in strong form to the adjoint equation
\begin{align}
-\frac{\partial p}{\partial t}-\mathrm{div}(k\nabla p)\hspace{.3mm}&=-(y-\overline{y}) \quad \text{in }\Omega \times [0,T)\label{adjoint1}\\
p\hspace{.3mm}&= 0\quad\text{in }\Omega \times \{T\}\label{adjoint2}\\
\left\llbracket k\nd{p}\right\rrbracket\hspace{.3mm}&=0 \quad\text{on }\Gi \times [0,T)\label{adjoint3}\\
\left\llbracket p\right\rrbracket\hspace{.3mm}&=0 \quad\text{on }\Gi \times [0,T)\label{adjoint4}\\
\nd{p}\hspace{.3mm}&=0\quad\text{on }\left(\Gb\cup\Gl\cup\Gr\right) \times [0,T)\label{adjoint5}\\
p\hspace{.3mm}&=0\quad\text{on }\Gt \times [0,T)\label{adjoint6}\\
p_1\hspace{.3mm}&=-k_1p\quad\text{on }\left(\Gb\cup\Gl\cup\Gr\right) \times [0,T)\label{adjoint7}\\
p_2\hspace{.3mm}&=k_1\nd{p}\quad\text{on }\Gt \times [0,T)\label{adjoint8}
\end{align}
and to the state equation
\begin{equation}
\td{y}-\mathrm{div}(k\nabla y)=f \quad \text{in }\Omega \times (0,T]. \label{design}
\end{equation}

Let $\Omega$ be fixed. Then it is easy to verify that
\begin{equation}
\label{minmax_formulation}
J(\Omega)= \min_{y\in \begin{rev}
W\left(0,T;H^1(\Omega)\right)
\end{rev}} \max_{p\in \begin{rev}W\left(0,T;H^1(\Omega)\right)\end{rev}} \LL (\Omega,y,p).
\end{equation}

Now, we formulate the following theorem which provides the representation of the shape derivative expressed as a domain integral. This domain integral will later allow us to calculate the boundary expression of the shape derivative.
\begin{theorem}
\label{theorem_sg_b}
Assume that the parabolic PDE problem (\ref{oc2}-\ref{oc6}) is $H^1$-regular, so that its solution $y$ is at least in $\begin{rev}W\left(0,T;H^1(\Omega)\right)\end{rev}$. Moreover, assume that the adjoint equation (\ref{adjoint1}-\ref{adjoint6}) admits a solution $p\in \begin{rev}W\left(0,T;H^1(\Omega)\right)\end{rev}$. Then the shape derivative of the objective $J$ (without perimeter regularization) at $\Omega$ in the direction $V$ is given by
\begin{equation}
\label{boundary_expression}
\boxed{
\begin{split}
dJ(\Omega)[V]=\int_{0}^{T}\int_{\Omega}&-k\nabla y^T\left(\nabla V+\nabla V^T\right)\nabla p-p\nabla f^T V\\
&+\mathrm{div}(V)\left(\frac{1}{2}(y-\overline{y})^2+\td{y}p+k\nabla y^T\nabla p-fp\right)dx\hspace{.3mm}dt
\end{split}
}
\end{equation}
\end{theorem}

\begin{proof}
We apply the theorem of Correa and Seger on the right hand side of (\ref{minmax_formulation}), i.e.\ we obtain formula (\ref{boundary_expression}) by evaluation of the shape derivative of the Lagrangian (\ref{lagrangian}) in its saddle point. The verification of the assumptions of this theorem can be checked in much the same way as in \cite[chapter 10, subsection 6.4]{Delfour-Zolesio-2001}. We leave it to the reader to verify them.
Applying the rule for differentiating domain integrals which is given in (\ref{der_domain_int}) yields
\begin{align}
&d\LL(\Omega,y,p)[V]\nonumber\\
&=\lim\limits_{s\to 0^+}\frac{\LL(\Omega_s,y,p)-\LL(\Omega,y,p)}{s}=\frac{d^+}{ds}\LL(\Omega_s,y,p)\,\rule[-3mm]{.1mm}{6mm}_{\hspace{.5mm}s=0}\nonumber\\
&=\int_\Omega\Bigg[ \hspace{1mm}\frac{1}{2}\intdt{D_m\left((y-\overline{y})^2\right)}+D_m\left(y(T,x)p(T,x)\right)-D_m\left(y_0\hspace{.3mm}p(0,x)\right)\nonumber\\
&\hspace*{13.8mm}-\intdt{D_m\left(\td{p}y\right)}+\intdt{D_m\left(k\nabla y^T\nabla p\right)}-\intdt{D_m\left(fp\right)}\nonumber\\
&\hspace*{13.8mm}+\mathrm{div}(V)\Bigg(\frac{1}{2}\intdt{(y-\overline{y})^2}+y(T,x)p(T,x)-y_0\hspace{.3mm}p(0,x)\nonumber\\
&\hspace*{30.7mm}-\intdt{\td{p}y}+\intdt{k\nabla y^T\nabla p}-\intdt{fp}\Bigg)\Bigg]dx\nonumber\\
&\hspace*{6.7mm}-\intdtdsi{D_m\left(\left\llbracket k \nd{y}p\right\rrbracket\right)+\mathrm{div}_\Gi(V)\hspace{.3mm}\left\llbracket k \nd{y}p\right\rrbracket}\nonumber\\
&\hspace*{6.7mm}-\intdtdso{D_m\left( k_1 \nd{y}p\right)+\mathrm{div}_\Go(V)\hspace{.3mm} k_1 \nd{y}p}\nonumber\\
&\hspace*{6.7mm}-\int_{\Gt}\intdt{ D_m\hspace{-.5mm}\left(p_1(y-1)\right)+\mathrm{div}_\Gt(V)\hspace{.3mm} p_1(y-1)}\hspace{.7mm}ds\nonumber\\
&\hspace*{6.7mm}-\int_{\Go\setminus\Gt}\intdt{ D_m\hspace{-.5mm}\left(p_2\nd{y}\right)+\mathrm{div}_{\Go\setminus\Gt}(V)\hspace{.3mm} p_2\nd{y}}\hspace{.7mm}ds\nonumber
\end{align}
Now, applying (\ref{product_rule}) and (\ref{material_grad_grad}) we obtain
\begin{align}
& d\LL(\Omega,y,p)[V]\nonumber\\
&=\int_\Omega\Bigg[ \hspace{1mm}\intdt{(y-\overline{y})\dot{y}}+\dot{y}(T,x)p(T,x)+y(T,x)\dot{p}(T,x)-y_0\hspace{.3mm}\dot{p}(0,x)\nonumber\\
&\hspace*{14.1mm}-\intdt{D_m\hspace{-.5mm}\left(\td{p}\right) y+\td{p}\dot{y}}-\intdt{\dot{f}p+f\dot{p}}\nonumber\\
&\hspace*{14.1mm}+\intdt{k\nabla \dot{y}^T\nabla p+k\nabla y^T\nabla \dot{p}-k\nabla y^T\left(\nabla V+\nabla V^T\right)\nabla p}\nonumber\\
&\hspace*{14.1mm}+\mathrm{div}(V)\Bigg(\frac{1}{2}\intdt{(y-\overline{y})^2}+y(T,x)p(T,x)-y_0\hspace{.3mm}p(0,x)\nonumber\\
&\hspace*{32mm}-\intdt{\td{p}y}+\intdt{k\nabla y^T\nabla p}-\intdt{fp}\Bigg)\Bigg]dx\nonumber\\
&\hspace*{8mm}-\intdtdsi{\left\llbracket D_m\hspace{-.5mm}\left(k\nd{y}\right)p+k\nd{y}\dot{p}\right\rrbracket+\mathrm{div}_\Gi(V)\hspace{.3mm}\left\llbracket k \nd{y}p\right\rrbracket}\nonumber
\end{align}
\begin{align}
&\hspace*{8mm}-\intdtdso{ D_m\hspace{-.5mm}\left(k_1\nd{y}\right)p+k_1\nd{y}\dot{p}+\mathrm{div}_\Go(V)\hspace{.3mm} k_1 \nd{y}p}\nonumber\\
&\hspace*{8mm}-\int_{\Gt}\intdt{ \dot{p}_1(y-1)+p_1\dot{y}+\mathrm{div}_\Gt(V)\hspace{.3mm} p_1(y-1)}\hspace{.7mm}ds\nonumber\\
&\hspace*{8mm}-\int_{\Go\setminus\Gt}\intdt{ \dot{p}_2\nd{y}+p_2D_m\hspace{-.5mm}\left(\nd{y}\right)+\mathrm{div}_{\Go\setminus\Gt}(V)\hspace{.3mm} p_2\nd{y}}\hspace{.7mm}ds\nonumber
\end{align}
From this we get
\begin{align}
& d\LL(\Omega,y,p)[V]\nonumber\\
&= \intdtdx{\left((y-\overline{y})-\td{p}-\mathrm{div}(k\nabla p)\right)\dot{y}+\left(\td{y}-\mathrm{div}(k\nabla y)-f\right)\dot{p}}\nonumber\\
&\hspace*{3.5mm}+\int_{0}^{T}\int_{\Omega}-k\nabla y^T\left(\nabla V+\nabla V^T\right)\nabla p-\dot{f}p\nonumber\\
&\hspace*{18mm}+\mathrm{div}(V)\left(\frac{1}{2}(y-\overline{y})^2+\td{y}p+k\nabla y^T\nabla p-fp\right)dx\hspace{.3mm}dt\nonumber\\
&\hspace*{3.5mm}+\intdtdsi{\left\llbracket k\nd{p}\dot{y}\right\rrbracket-\left\llbracket D_m\hspace{-.5mm}\left(k\nd{y}\right)p\right\rrbracket-\mathrm{div}_\Gi(V)\left\llbracket k\nd{y}p\right\rrbracket}\nonumber\\
&\hspace*{3.5mm}+\intdtdso{ \nd{p}\dot{y}-k_1D_m\hspace{-.5mm}\left(k_1\nd{y}\right)p-\mathrm{div}_\Go(V)k_1\nd{y}p}\nonumber\\
&\hspace*{3.5mm}-\int_{\Gt}\intdt{ \dot{p}_1(y-1)+p_1\dot{y}+\mathrm{div}_\Gt(V)\hspace{.3mm} p_1(y-1)}\hspace{.7mm}ds\nonumber\\
&\hspace*{3.5mm}-\int_{\Go\setminus\Gt}\intdt{ \dot{p}_2\nd{y}+p_2D_m\hspace{-.5mm}\left(\nd{y}\right)+\mathrm{div}_{\Go\setminus\Gt}(V)\hspace{.3mm} p_2\nd{y}}\hspace{.7mm}ds\label{derLag}
\end{align}
where the term $\dot{f}p$ is equal to $p\nabla f^TV$ due to (\ref{shape_der}). The outer boundary $\Go$ is not variable. Thus, we can choose the deformation vector field $V$ equals zero in small neighbourhoods of $\Go$. Moreover, each material derivative in small neighbourhoods of $\Go$ is equal to zero. Therefore, the three outer integrals in (\ref{derLag}) vanish. Now, let us consider the saddle point condition (\ref{saddlepointcond}) or respectively (\ref{adjoint1}-\ref{design}). Due to the continuity of the state and of the flux (\ref{oc6}) their material derivative is continuous. Thus, we get
\begin{align}
\left\llbracket k\nd{p}\dot{y}\right\rrbracket&=\dot{y}\left\llbracket k\nd{p}\right\rrbracket\stackrel{(\ref{adjoint3})}=0\quad \text{on }\Gi\\
\left\llbracket D_m\hspace{-.5mm}\left(k\nd{y}\right)p\right\rrbracket&=D_m\hspace{-.5mm}\left(k\nd{y}\right)\left\llbracket p\right\rrbracket\stackrel{(\ref{adjoint4})}=0\quad \text{on }\Gi.
\end{align}
Then
\begin{equation}
\label{pbe}
\left\llbracket k\nd{y}p\right\rrbracket=0\quad \text{on }\Gi
\end{equation}
follows from (\ref{oc6}), (\ref{adjoint4}) and the identity
\begin{equation}
\left\llbracket ab\right\rrbracket=\left\llbracket a\right\rrbracket b_1 +a_2 \left\llbracket b\right\rrbracket= a_1 \left\llbracket b\right\rrbracket+\left\llbracket a\right\rrbracket b_2
\end{equation}
which implies
\begin{equation}
\left\llbracket ab\right\rrbracket=0 \text{ if } \left\llbracket a\right\rrbracket=0\wedge \left\llbracket b\right\rrbracket=0.
\end{equation}
By combining (\ref{derLag}-\ref{pbe}), we obtain (\ref{boundary_expression}).
\end{proof}

Now, we want to convert the domain integral (\ref{boundary_expression}) into a boundary integral which is better suited for a finite element implementation as already mentioned for example in \cite[remark 2.3, p. 531]{Delfour-Zolesio-2001}. The following theorem is a generalization of lemma 1 in \cite{Paganini} for parabolic problems and provides two representations of the shape derivative expressed as a boundary integral.


\begin{theorem}
\label{theorem_sg_bb}
Under the assumptions of theorem \ref{theorem_sg_b} the shape derivative of the objective $J$ (without perimeter regularization) at $\Omega$ in the direction $V$ is given by
\begin{equation}
\boxed{
dJ(\Omega)[V]= \intdsidt{\left<V,n\right>\left\llbracket -2k\nd{y}\nd{p}+k\nabla y^T\nabla p\right\rrbracket}\label{shape_der1}
}
\end{equation}
Let $y_1:=y\,\rule[-2mm]{.1mm}{4mm}_{\hspace{.5mm}\Omega_1}$ and $p_2:=p\,\rule[-2mm]{.1mm}{4mm}_{\hspace{.5mm}\Omega_2}$. Then the shape derivative of the objective $J$ at $\Omega$ in the direction $V$ can be expressed as
\begin{equation}
\boxed{
dJ(\Omega)[V]= \intdsidt{\left\llbracket k\right\rrbracket \nabla y_1^T\nabla p_2\left<V,n\right>}\label{shape_der2}
}
\end{equation}
\end{theorem}

\begin{proof}
Integration by parts on the integral $$\int_\Omega \mathrm{div}(V)\left(\frac{1}{2}(y-\overline{y})^2+\td{y}p+k\nabla y^T\nabla p-fp\right)dx$$ in (\ref{boundary_expression}) yields
\begin{align}
&\int_\Omega \mathrm{div}(V)\left(\frac{1}{2}(y-\overline{y})^2+\td{y}p+k\nabla y^T\nabla p-fp\right)dx\nonumber\\
&=\int_{\Omega_1} \mathrm{div}(V)\left(\frac{1}{2}(y-\overline{y})^2+\td{y}p+k_1\nabla y^T\nabla p-fp\right)dx\nonumber\\
&\hspace*{3.5mm}+\int_{\Omega_2} \mathrm{div}(V)\left(\frac{1}{2}(y-\overline{y})^2+\td{y}p+k_2\nabla y^T\nabla p-fp\right)dx\nonumber\\
&=\int_{\Gi\cup\Go}\left(\frac{1}{2}(y-\overline{y})^2+\td{y}p+k_1\nabla y^T\nabla p-fp\right)\left<V,n\right>ds\nonumber\\
&\hspace*{3.5mm}-\int_{\Omega_1}V^T\left((y-\overline{y})\nabla y+\nabla\left(\td{y}p\right)+k_1\nabla\left(\nabla y^T\nabla p\right)-\nabla\left(fp\right)\right)\hspace{.3mm}dx\nonumber\\
&\hspace*{3.5mm}+\int_{\Gi} \left(\frac{1}{2}(y-\overline{y})^2+\td{y}p+k_2\nabla y^T\nabla p-fp\right)\left<V,-n\right>ds\nonumber\\
&\hspace*{3.5mm}-\int_{\Omega_2} V^T\left((y-\overline{y})\nabla y+\nabla\left(\td{y}p\right)+k_2\nabla\left(\nabla y^T\nabla p\right)-\nabla\left(fp\right)\right)\hspace{.3mm}dx\nonumber\\
&=-\int_{\Omega}V^T\left((y-\overline{y})\nabla y+\nabla\left(\td{y}p\right)+k\nabla\left(\nabla y^T\nabla p\right)-\nabla f p-f\nabla p\right)\hspace{.3mm}dx\nonumber\\
&\hspace*{3.5mm}+\int_{\Gi} \left\llbracket\left(\frac{1}{2}(y-\overline{y})^2+\td{y}p+k\nabla y^T\nabla p-fp\right)\left<V,n\right>\right\rrbracket ds\nonumber\\
&\hspace*{3.5mm}+\int_{\Go}\left(\frac{1}{2}(y-\overline{y})^2+\td{y}p+k_1\nabla y^T\nabla p-fp\right)\left<V,n\right>ds\label{int_by_parts1}
\end{align}
Combining (\ref{boundary_expression}), (\ref{int_by_parts1}) and the vector calculus identity 
\begin{equation*}
\nabla y^T\left(\nabla V+\nabla V^T\right)\nabla p+V^T\nabla\left(\nabla y^T\nabla p\right)=\nabla p^T\nabla\left(V^T\nabla y\right)+\nabla y^T \nabla\left(V^T\nabla p\right)
\end{equation*}
which was proved in \cite{Berggren} gives
\begin{align}
& dJ(\Omega,y,p)\nonumber\\
&=\int\limits_{0}^{T} \Bigg[\int_{\Omega}-k\nabla p^T\nabla\left(V^T\nabla y\right)-k\nabla y^T \nabla\left(V^T\nabla p\right)-(y-\overline{y})V^T\nabla y\nonumber\\
&\hspace*{1.6cm}-V^T\nabla \left(\td{y}p\right)+fV^T\nabla p \hspace{.3mm}dx\nonumber\\
&\hspace*{1cm}+\int_{\Gi} \left\llbracket\left(\frac{1}{2}(y-\overline{y})^2+\td{y}p+k\nabla y^T\nabla p-fp\right)\left<V,n\right>\right\rrbracket ds\nonumber\\
&\hspace*{1cm}+\int_{\Go}\left(\frac{1}{2}(y-\overline{y})^2+\td{y}p+k_1\nabla y^T\nabla p-fp\right)\left<V,n\right>ds\Bigg]dt. \label{sg1}
\end{align}
Then, applying integration by parts on the integral $\int_{\Omega}k\nabla y^T\nabla\left(V^T\nabla p\right)dx$ in (\ref{sg1}) we get
\begin{align}
&\int_{\Omega}k\nabla y^T\nabla\left(V^T\nabla p\right)dx\nonumber\\
&=\int_{\Omega_1}k_1\nabla y^T\nabla\left(V^T\nabla p\right)dx+\int_{\Omega_2}k_2\nabla y^T\nabla\left(V^T\nabla p\right)dx\nonumber\\
&=\int_{\Gi\cup\Go}k_1\nd{y}V^T\nabla p \hspace{.3mm}ds-\int_{\Omega_1}\mathrm{div}(k_1\nabla y)\nabla p^TVdx+\int_{\Gi}k_2\nd{y}V^T\nabla p \hspace{.3mm}ds\nonumber\\
&\hspace{3.7mm}-\int_{\Omega_2}\mathrm{div}(k_2\nabla y)\nabla p^TVdx\nonumber\\
&=-\intdx{\mathrm{div}(k\nabla y)\nabla p^TV}+\int_{\Gi}\left\llbracket k\nd{y}V^T\nabla p \hspace{.3mm}\right\rrbracket ds+\int_{\Go}k_1\nd{y}V^T\nabla p \hspace{.3mm}ds\label{int_by_parts2}
\end{align}
and analogously
\begin{align}
& \int_{\Omega}k\nabla p^T\nabla\left(V^T\nabla y\right)dx\nonumber\\
&=-\intdx{\mathrm{div}(k\nabla p)\nabla y^TV}+\int_{\Gi}\left\llbracket k\nd{p}V^T\nabla y \hspace{.3mm}\right\rrbracket ds+\int_{\Go}k_1\nd{p}V^T\nabla y \hspace{.3mm}ds. \label{int_by_parts3}
\end{align}
Integration by parts on the integral $\intdt{\nabla \td{y}p}$ in (\ref{sg1}) yields
$$\intdt{\nabla \td{y}p}=\nabla y(T,x)\hspace{.3mm}p(T,x)-\intdt{\nabla y \td{p}}$$
Thus, it follows that
\begin{align}
& dJ(\Omega,y,p)\nonumber\\
&= \int\limits_{0}^{T} \Bigg[\int_\Omega \nabla p^T V \left(-\td{y}+\mathrm{div}(k\nabla y)+f\right)\nonumber\\
&\hspace*{1.6cm}+\nabla y^TV\left(\td{p}+\mathrm{div}(k\nabla p)-(y-\overline{y})\right)dx\nonumber\\
&\hspace*{1cm}+\int_{\Gi} \left\llbracket\left(\frac{1}{2}(y-\overline{y})^2+\td{y}p-k\nabla y^T\nabla p-fp\right)\left<V,n\right>\right\rrbracket\nonumber\\
&\hspace*{2.2cm}-\left\llbracket k\nd{y}V^T\nabla p\hspace{.3mm}\right\rrbracket -\left\llbracket k\nd{p}V^T\nabla y\hspace{.3mm}\right\rrbracket ds\Bigg]dt\nonumber\\
&\hspace*{1cm}+\int_{\Go}\left(\frac{1}{2}(y-\overline{y})^2+\td{y}p-k_1\nabla y^T\nabla p-fp\right)\left<V,n\right>\nonumber\\
&\hspace*{2.2cm}-k_1\nd{y}V^T\nabla p-k_1\nd{p}V^T\nabla y\hspace{.5mm}ds\Bigg]dt\nonumber\\
&\hspace*{3.5mm}+\intdx{V^T\nabla y(T,x)\hspace{.3mm}p(T,x)}\label{sg2}
\end{align}
The domain integrals in (\ref{sg2}) vanish due to (\ref{adjoint1}), (\ref{adjoint2}) and (\ref{design}). Moreover, the term $\left\llbracket \left(\td{y}-f\right)p\right\rrbracket$ vanishes because of (\ref{adjoint4}) and the term $\left\llbracket \frac{1}{2}(y-\overline{y})^2\right\rrbracket$ because of (\ref{oc6}). Then
\begin{equation}
\label{pb}
\left\llbracket k\nd{y}V^T\nabla p\right\rrbracket=\left\llbracket k\nd{p}V^T\nabla y\right\rrbracket =\left<V,n\right>\left\llbracket k\nd{y}\nd{p}\right\rrbracket
\end{equation}
follows from (\ref{oc6}) and (\ref{adjoint3}). Since the outer boundary $\Go$ is not variable, we can choose the deformation vector field $V$ equals zero in small neighbourhoods of $\Go$. Therefore, the outer integral in (\ref{sg2}) disappears and we obtain the interface integral (\ref{shape_der1}). It is easy to verify that
\begin{equation}
\label{id_sg}
\int_\Gi\left<V,n\right>\left\llbracket -2k\nd{y}\nd{p}+k\nabla y^T\nabla p\right\rrbracket ds=\int_\Gi\left\llbracket k\right\rrbracket \nabla y_1^T\nabla p_2\left<V,n\right>ds
\end{equation}
which completes the proof. For a detailed computation of (\ref{pb}) and (\ref{id_sg}) we refer the reader to \cite[p. 320]{ItoKunisch}.
\end{proof}

Now, we consider the objective (\ref{oc1}) with perimeter regularization. For the finite element implementation of (\ref{oc1}--\ref{oc6}) in section \ref{numex} we need a representation of its shape derivative expressed as boundary integral. Two such representations are given by the following theorem. 
\begin{theorem}\label{bc}
Under the assumptions of theorem \ref{theorem_sg_b} the shape derivative of the objective $J$ (with perimeter regularization) at $\Omega$ in the direction $V$ is given by
\begin{equation}
\label{sd1}
\boxed{
dJ(\Omega)[V]=\int_{\Gi}\left[\intdt{\left<V,n\right>\left\llbracket -2k\nd{y}\nd{p}+k\nabla y^T\nabla p\right\rrbracket}+\left<V,n\right>\mu\kappa\right]ds
}
\end{equation}
where $\kappa$ denotes the curvature corresponding to the normal $n$. Let $y_1:=y\,\rule[-2mm]{.1mm}{4mm}_{\hspace{.5mm}\Omega_1}$ and $p_2:=p\,\rule[-2mm]{.1mm}{4mm}_{\hspace{.5mm}\Omega_2}$. Then, the shape derivative of the objective $J$ (with perimeter regularization) at $\Omega$ in the direction $V$ can be expressed as
\begin{equation}
\label{sd2}
\boxed{
dJ(\Omega)[V]=\int_{\Gi}\left[\intdt{\left<V,n\right>\left\llbracket k\right\rrbracket \nabla y_1^T\nabla p_2}+\left<V,n\right>\mu\kappa\right]ds
}
\end{equation}
\end{theorem}
\begin{proof}
Combining theorem \ref{theorem_sg_bb} with proposition 5.1 in \cite{Novruzi-2002} we get (\ref{sd1}) and (\ref{sd2}).
\end{proof}
\begin{rev}
\begin{remark}
Throughout the derivation of theorems
\ref{theorem_sg_b}, \ref{theorem_sg_bb} and \ref{bc} above, we have tacitly assumed shape differentiability. Without this property, the formula manipulations can only be understood formally. The key issue  is the continuity of trace mappings of the state $y$, the adjoint $p$ and their first derivatives in theorem \ref{theorem_sg_bb} as mappings to integrable functions on the interface $\Gi$. Because of the jump in the diffusion coefficient and since we assume for the observation $\bar{y}\in W\left(0,T;H^1(\Omega)\right)$, as mentioned above, we can only assume $y,p\in W\left(0,T;H^1(\Omega)\right)$, which seems to be problematic in relation to integrable traces of derivatives. However, we can generalize the discussion in \cite{ItoKunisch} for the elliptic version of our parabolic shape optimization problem in a straight forward manner. This shows that indeed $y|_{\Omega_i},p|_{\Omega_i}\in W\left(0,T;H^2(\Omega_{i})\right)$, for $i=1,2$, which means that the trace mapping is also continuous for the first derivatives and thus yields shape differentiability
\end{remark}
\end{rev}

\section{A quasi-Newton approach on shape manifolds}\label{quasi}
As pointed out in \cite{VHS-shape-Riemann}, shape optimization can be viewed as optimization on Riemannian shape manifolds and resulting optimization methods can be constructed and analyzed within this framework, which combines algorithmic ideas from \cite{Absil-book-2008} with the differential geometric point of view established in \cite{MM-2006}.
\begin{rev}
As in \cite{VHS-shape-Riemann}, we study connected and compact subsets $\Omega_2$ of $\R^2$ with $\Omega_2\neq\emptyset$ and $C^\infty$ boundary $\partial\Omega_2$ (cf. figure \ref{fig_Omega}).\end{rev}
We now identify the variable boundary $\partial \Omega_2 = \Gi$ with a simple closed curve $c\colon S^1 \to \mathbbm{R}^2$.
Additionally, we need to describe a space including all feasible shapes $\Gi$ and the corresponding tangent spaces.
In \cite{MM-2006}, this set of smooth boundary curves $c$ is characterized by  
\[
B_e(S^1,\R^2):=\mbox{Emb}(S^1,\R^2)/\mbox{Diff}(S^1)
\] 
i.e., as the set of all equivalence classes of $C^\infty$ embeddings of $S^1$ into the plane ($\mbox{Emb}(S^1,\R^2)$), where the equivalence relation is defined by the set of all $C^\infty$ re-parameterizations, i.e., diffeomorphisms of $S^1$ into itself ($\mbox{Diff}(S^1)$). A particular point on the manifold $B_e(S^1,\R^2)$ is represented by a curve $c:S^1\ni\theta\mapsto c(\theta)\in\R^2$. Because of the equivalence relation ($\mbox{Diff}(S^1)$), the tangent space is isomorphic to the set of all normal $C^\infty$ vector fields along $c$, i.e.
\[
T_cB_e\cong\{h\ |\ h=\alpha n,\, \alpha\in C^\infty(S^1,\R)\}
\]
where $n$ is the unit exterior normal field of the shape $\Omega_2$ defined by the boundary $\partial\Omega_2=c$ such that $n (\theta)\perp c^\prime$ for all $\theta\in S^1$ and $c^\prime$ denotes the circumferential derivative as in \cite{MM-2006}. For our discussion, we pick among the other metrics discussed in \cite{MM-2006} the Sobolev metric family for
$A\ge 0$
\begin{align*}
g^1:\ &T_cB_e\times T_cB_e\to \R\\
    & (h,k)\mapsto
    \int\limits_{c=\partial\Omega_2}\alpha\beta+A\alpha^\prime\beta^\prime
    ds =((id-A\triangle_c)\alpha,\beta)_{L^2(c)}
\end{align*}
where $h=\alpha n$ and $k=\beta n$ denote two elements from
the tangent space at $c$ and $\triangle_c$ denotes the Laplace-Beltrami operator on the surface $c$.  In \cite{MM-2006} it is shown that for $A>0$ the scalar product $g^1$ defines a Riemannian metric on $B_e$ and 
thus, geodesics can be used to measure distances.
Unfortunately, this is not the case for the most simple member $g^0$ of the metric family $g^1$, where $A=0$.

\begin{rev}
With the shape space $B_e$ and its tangent space in hand we can now form the Riemannian shape gradient corresponding to a shape derivative given in the form
\[
dJ[V]=\int_c \gamma\scp{V,n}ds.
\]
In our setting the shape derivative is given in theorem \ref{theorem_sg_bb} or \ref{bc} and the Riemannian metric by $g^1$.
Finally, the Riemannian shape gradient $\text{grad}J$ is obtained by
\[
\grad J=gn\, \quad \mbox{with } (id -A\triangle_c)\tilde{\gamma}=\gamma\, .
\]
\end{rev}

In the sequel, we will also need the concept of the covariant derivative $\nabla$ and of the exponential map
\begin{align*}
\exp_c:\ &T_cB_e\to B_e\\
    & h\mapsto \exp_c(h)
\end{align*}
defining a local diffeomorphism between the tangent space and the manifold by following the locally uniquely defined geodesic starting in $c\in B_e$ with velocity $h\in T_cB_e$. The exponential map depends on the Riemannian metric $g^1$ in the usual way. 

The application of quasi-Newton methods is based on the secant condition, which is formulated on the Riemannian manifold $B_e$ analogously to  \cite{Absil-book-2008} for a step $c_{j+1}:=R_{c_j}(\eta)$ resulting from an increment $\eta_j\in T_{c_j}B_e$ in iteration $j$ via a retraction $R$ as
\[
\grad J (c_{j+1})-{\cal T}_{\eta_{j}}\grad J (c_{j})=G_{j+1}[{\cal T}_{\eta_{j}}\eta_{j}]
\]
where ${\cal T}:TB_e\oplus TB_e\to TB_e: (h_c,k_c)\mapsto {\cal T}_{h_c}k_c$ is a vector transport associated to the retraction $R$ and $G_{j+1}$ is intended to approximate the
Riemannian Hessian $\nabla\grad J(c_{j+1})$. In order to formulate the BFGS-update in a concise way, we need to introduce the following notation for a typical linear operator associated with the Riemannian metric
\begin{align*}
h\otimes k:\ &T_cB_e\to T_cB_e\\
    & v\mapsto g^1(k,v)h
\end{align*}
with this notation and together with the following abbreviations
\begin{align*}
s_j:=&{\cal T}_{\eta_{j}}\eta_{j}\in T_{c_{j+1}}B_e\\
y_j:=& \grad J (c_{j+1})-{\cal T}_{\eta_{j}}\grad J (c_{j})\in T_{c_{j+1}}B_e
\end{align*}
we can rephrase the BFGS-update on Riemannian shape space endowed with the metric $g^1$ as
\[
G_{j+1}=\tilde{G}_j-\frac{(\tilde{G}_js_j)\otimes(\tilde{G}_js_j)}{g^1(s_j,\tilde{G}_js_j)}+\frac{y_j\otimes y_j}{g^1(s_j,y_j)}
\]
where $\tilde{G}_j:={\cal T}_{\eta_j}\circ G_j\circ {\cal
  T}_{\eta_j}^{-1}$. 
In \cite{Ring-Wirth-2012}, superlinear convergence properties for BFGS-quasi-Newton-methods on manifolds are analysed for the case that ${\cal T}_{\eta_j}$ is
an isometry. This requirement is satisfied, e.g., if $\cal T$ and $R$
are the parallel transport and the exponential map. It is well-known (e.g.~\cite{Nocedal-Wright}) that the corresponding update of the inverse operator can be written in the form
\begin{align*}
G_{j+1}^{-1} = \left( id - \frac{s_j\otimes y_j}{g^1(y_j,s_j)} \right) \tilde{G}_j^{-1} \left(  id - \frac{y_j \otimes s_j}{g^1(y_j,s_j)} \right) + \frac{s_j\otimes s_j}{g^1(y_j,s_j)}
\end{align*}
This is the most convenient update formulation in an infinite dimensional setting.
In standard formulation, update formulas require the storage of the whole convergence history up to the current iteration. Limited memory update techniques (e.g.~\cite{Nocedal-Wright}) have been developed, in order to reduce the amount of storage. In the current situation, this can be analogously formulated in the following algorithmic way:

\begin{algorithmic}
\State $\rho_j \gets g^1(y_j,s_j)^{-1}$
\State $q \gets \grad J(c_j)$
\For{$i = j-1, \dots , j-m$}
	\State $s_i \gets {\cal T}_{q}s_i$
	\State $y_i \gets {\cal T}_{q}y_i$
	\State $\alpha_i \gets \rho_i g^1(s_i,q)$
	\State $q \gets q - \alpha_i d_i$
\EndFor
\State $z \gets \grad J(c_j)$
\State $q \gets \frac{g^1(y_{j-1},s_{j-1})}{g^1(y_{j-1},y_{j-1})} \grad J(c_j)$
\For{$i = j-m, \dots , j-1$}
	\State $\beta_i \gets \rho_i g^1(y_i,z)$
	\State $q \gets q + (\alpha_i - \beta_i) s_i$
\EndFor \\
\Return $q = G_j^{-1} \grad J(c_j)$
\end{algorithmic}
\begin{rev}
This is conceptually similar to the double loop algorithm in finite dimensional Euclidean spaces.
Yet the inner products are now given by the Sobolev metric and vector transports have to be considered.
\end{rev}

\section{Numerical Results and implementation details}\label{numex}
We test the algorithms developed in the previous section with the problem (\ref{oc1}-\ref{oc3}) in the domain
$\Omega=[-1,1]^2$, which contains a compact and closed subset $\Omega_2$ with smooth boundary. The parameter $k_1$ is valid in the exterior $\Omega_1 = \Omega \setminus \Omega_2$ and the parameter $k_2$ is valid in the interior $\Omega_2$. First, we build artificial data $\bar{y}$, by solving the state equation for the setting
$\bar{\Omega}_2:=\{x:\, \norm{x}_2 \le r\}$ with $r=0.5$. Afterwards, we choose another initial domain $\Omega_1$ and $\Omega_2$. Figure \ref{fig-shape} illustrates the interior boundary $\Gamma_\text{int}$ around the initial domain $\Omega_2$ and the target domain $\bar{\Omega}_2$.

For this particular test case we choose the parameter to be $k_1 = 1$ and $k_1 = 0.001$ and a regularization parameter of $\mu = 0.0001$. The final time of the simulation is $T=20$.
In order to solve the boundary value problem (\ref{oc2}-\ref{oc6}), its weak form (\ref{wf}) is discretized in space using standard linear finite elements.
The parameter $k$ is approximated in a element-wise constant space.
Due to the choice of a continuous space for $y$ and a discontinuous space for $k$, conditions \eqref{oc6} are automatically fulfilled.
Furthermore, we choose the implicit Euler method for the temporal discretization.
The interval $[0, T]$ is therefore divided by $30$ equidistantly distributed time steps.
Due to the self adjoint nature of the problem we can solve the adjoint equation (\ref{adjoint1}-\ref{adjoint6}) applying the same spatial and temporal discretization as for the primal one.
Finally, the resulting linear systems are solved using the conjugate gradient method.

An essential part of this algorithm is a discrete version of the Laplace-Beltrami operator, which is on the one hand used to get a feasible representation of the shape gradient and on the other hand is needed for the scalar products in the BFGS method.
We therefore implement the formulas given in \cite{meyer2003discrete} which describe an operator that can be used both as the Laplace-Beltrami and to compute the discrete mean curvature.
However, this approach is tailored for two dimensional, triangulated surfaces.
We thus have to extend the polygonal line in our test case in the third coordinate direction such that a surface is spanned which is then triangulated.

\begin{figure}
\begin{center}
\includegraphics[width=0.5\textwidth]{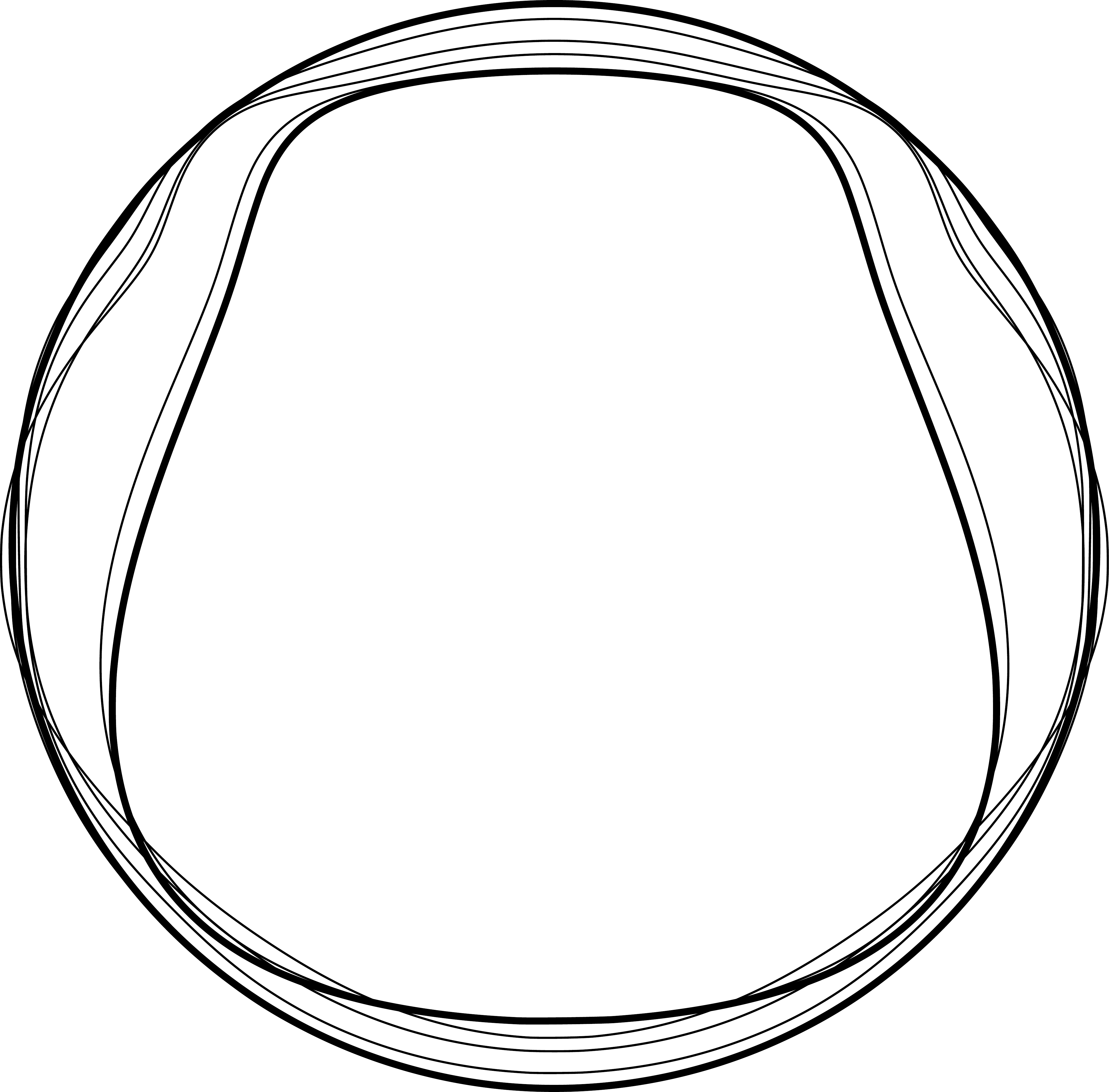}
\end{center}
\caption{\label{fig-shape}Initial and final shape geometry}
\end{figure}

\begin{figure}
\begin{center}
\includegraphics[width=1\textwidth]{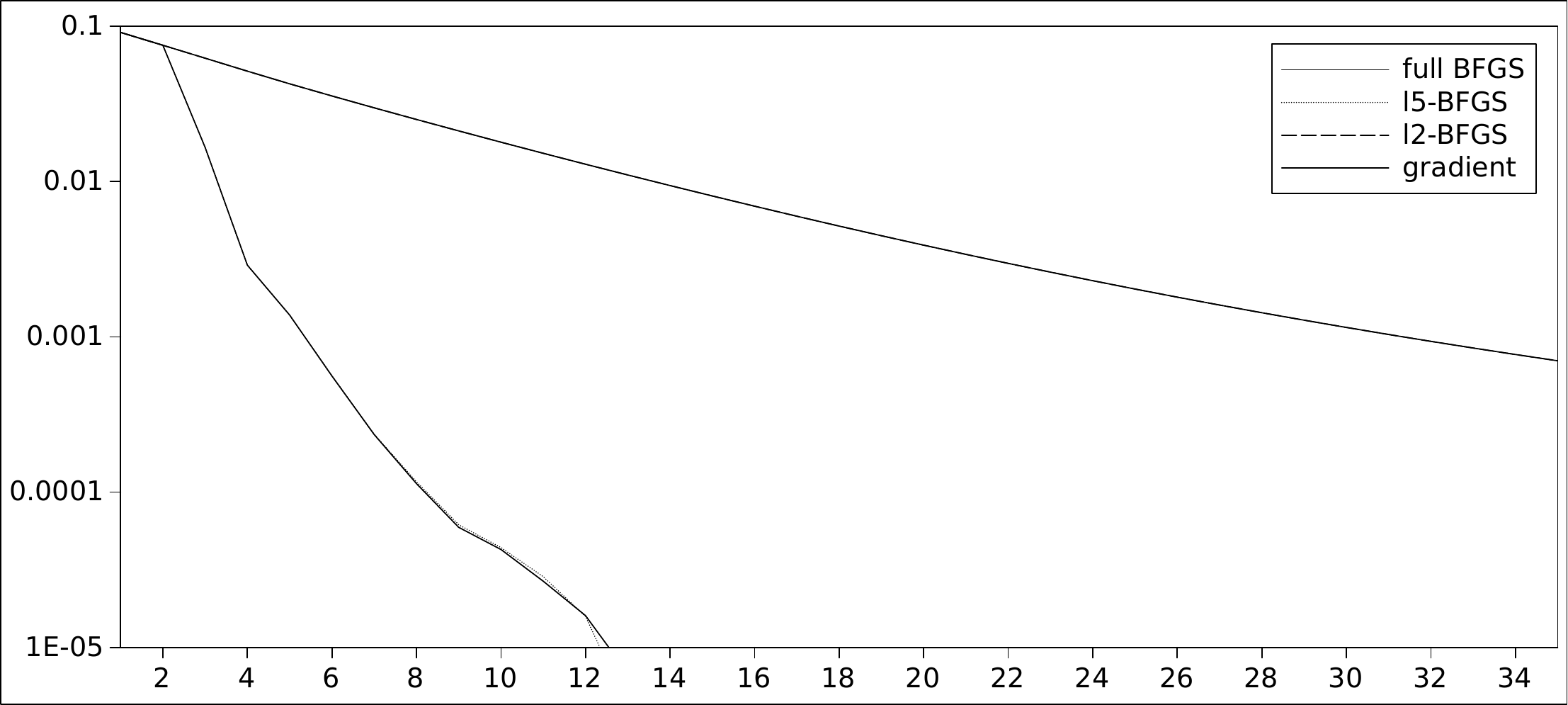}
\end{center}
\caption{\label{fig-convergence-parabolic}Different l-BFGS methods for the parabolic problem}
\end{figure}

We investigate the convergence behaviour of the following optimization
strategies
\begin{enumerate}
\item steepest descent method with fixed step-size 1.
\item limited memory BFGS quasi-Newton with constant metric parameter
  $A=0.001$ and also step-size 1.
\end{enumerate}
As observed below, the exact choice of $A$ has only a mild influence
on the overall convergence properties.

The necessary operations between the tangent spaces and the manifold
are chosen essentially as the identity operator, i.e., for $\eta\in
T_{c}B_e$, we define\
\[
R_{c}(\eta)(s):=s+\eta(s)\, ,\ \forall s\in c
\]
and
\[
{\cal T}_{\eta}v(s):=v(s-\eta(s)) \, ,\ \forall s\in R_{c}(\eta)
\]
This setting corresponds to one explicit Euler step for the exponential map and the parallel
transport in the case of the choice $A=0$ in the metric $g^1$. From an
implementation point of view this is most convenient. Computing an
explicit Euler step for the exponential map and parallel transport for
$A>0$ would require
the solution of yet another solution of an elliptic equation on the
surface to be optimized. However, numerical experiments have shown that
the convergence properties of the resulting iterations are not changed
and thus the additional numerical effort does not pay off in comparison with the
inexpensive retraction above.

A major problem, which arises in the discrete case using linear finite elements, is that both the representation of the shape gradient as computed in \eqref{shape_der1} or \eqref{shape_der2} and the normal vector field is discontinuous across element interfaces and can thus not be applied directly as a deformation to the shape.
We therefore solve the following $L^2$-projection to obtain a representation in piece-wise linear basis functions:
\begin{equation}
\int_{\Gi} u v \,ds = \int_{\Gi} \left( \intdt{\left\llbracket k\right\rrbracket \nabla y_1^T\nabla p_2} \right) n v \,ds
\end{equation}
for all linear test-functions $v$ on $\Gi$.
The resulting element-wise linear function $u$ can then be applied as a Dirichlet boundary condition in a linear elasticity equation.
A second Dirichlet condition is chosen to be zero at the outer boundary of $\Omega$ such that the domain keeps its outer shape.
Solving this PDE finally gives a deformation field which can be evaluated in each mesh node and gives a triangulation of the optimized shape without the need of remeshing the domain $\Omega$.

We do not apply a line search strategy in this setting because of the computational cost.
Each descent test in the line search requires the solution of the parabolic PDE in time and additionally the computation of the mesh deformation which includes also a PDE.
Since the resulting step lengths in both the gradient method and BFGS are feasible for this particular setting, a line search is not obligatory.

\begin{figure}
\begin{center}
\includegraphics[width=1\textwidth]{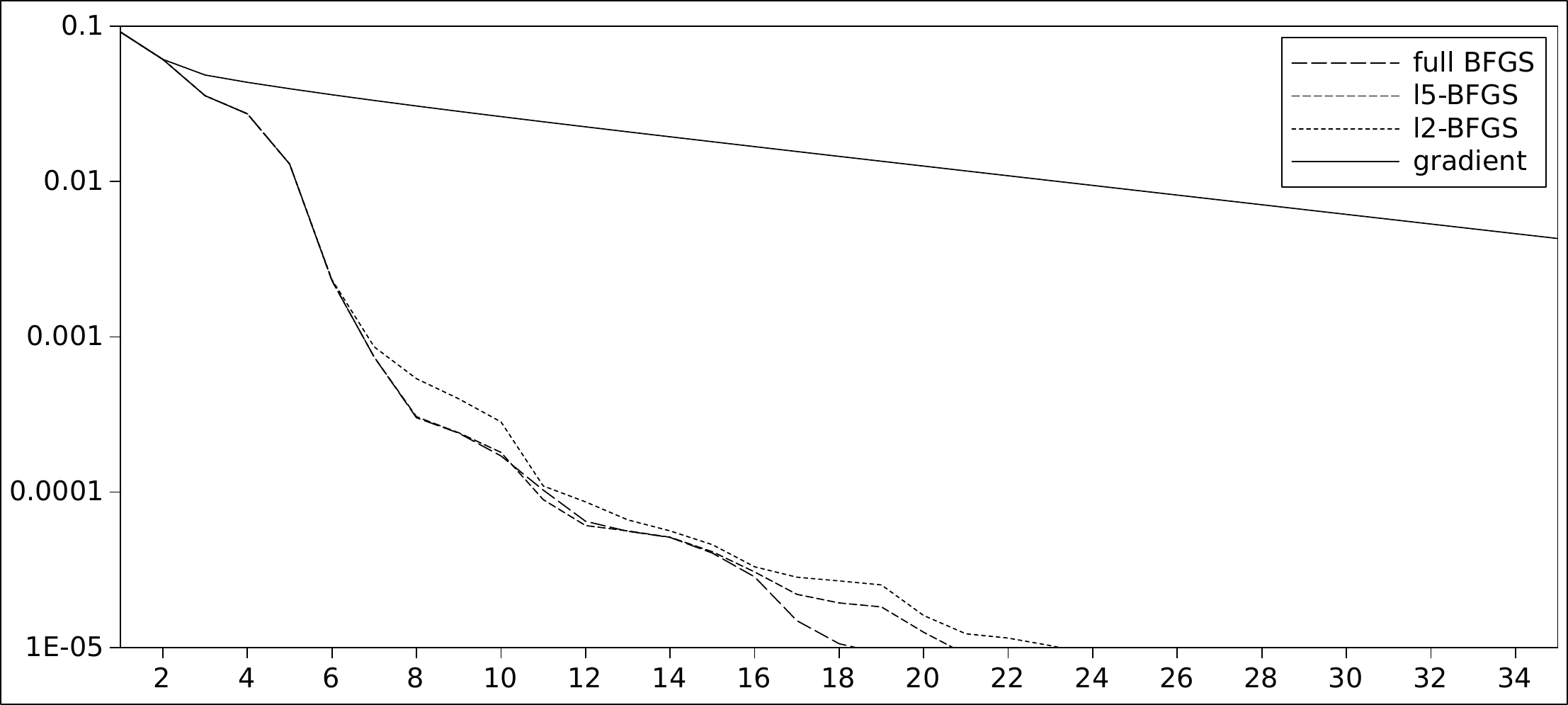}
\end{center}
\caption{\label{fig-convergence-elliptic}Different l-BFGS methods for the elliptic problem}
\end{figure}

\begin{figure}
\begin{center}
\includegraphics[width=1\textwidth]{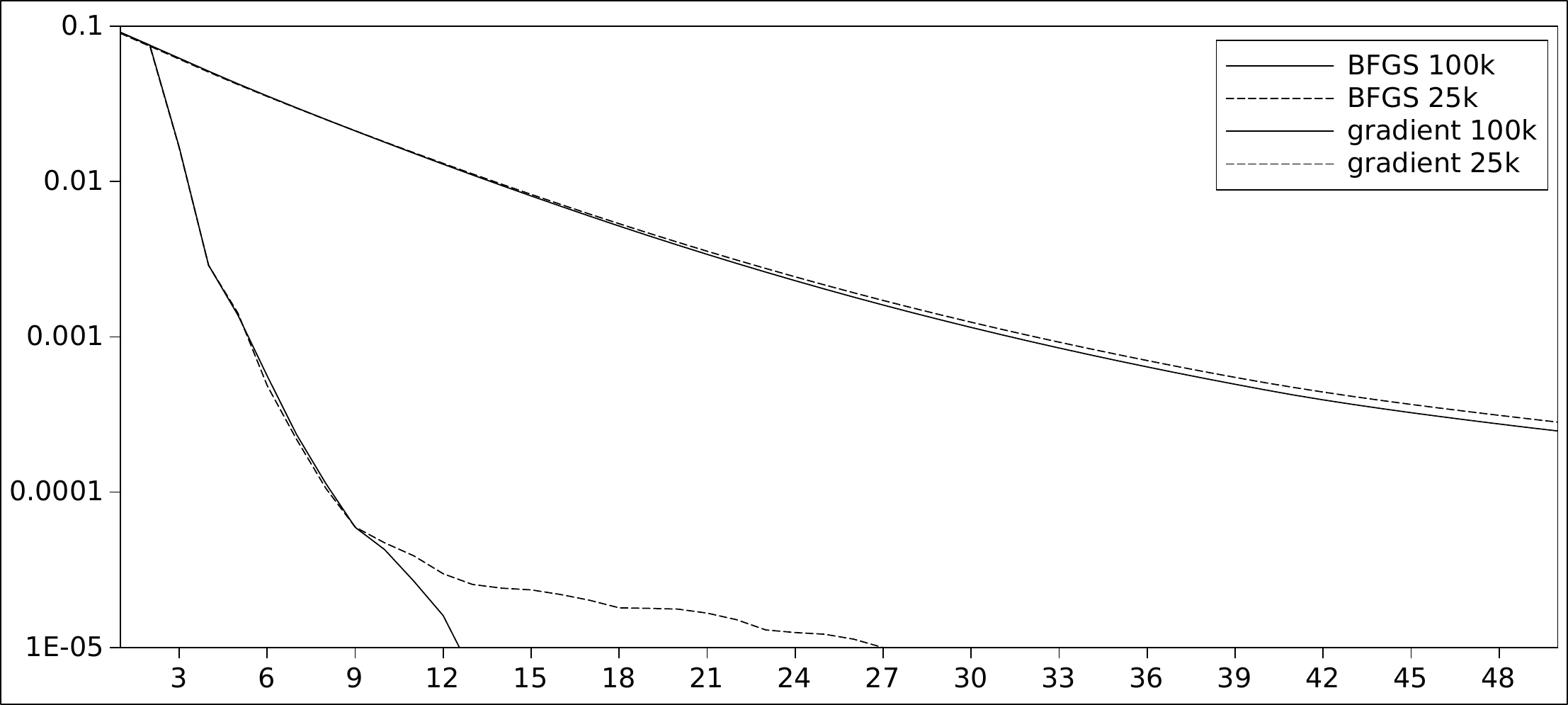}
\end{center}
\caption{\label{fig-different-grids}Comparison of BFGS and gradient method on different grids}
\end{figure}

The measurements of convergence rates ideally has to be performed in terms of the geodesic distance $\delta(c^k,\hat{c})$, where $\hat{c}$ denotes the optimal solution, as specified in \cite{BHM-2011}. However, this would require the computation of the full geodesic connecting the current iterate with the solution, which is a highly expensive operation. Because of the rigidity of the retraction, a first order approximation is
\[
\delta(c_j,\hat{c})\doteq \|\eta\|
\]
where $\eta\in T_{\hat{c}}B_e$ is defined by $c^k=R_{\hat{c}}(\eta)$ and $\|\eta\|=g^1(\eta,\eta)^{1/2}$.

\begin{rev}
In the discrete setting we therefore compute for each node of the iterated shape $c_j$ the shortest distance to $\hat{c}$ in normal direction. We then form the $L^2$-Norm of this distance field over $\hat{c}$, which is used to measure the convergence. It should be mentioned that the cost of this operation is quadratic with respect to the number of nodes on the surface. Starting in one node on $c_j$ in normal direction, the determination of a point of intersection with $\hat{c}$ requires to check all boundary segments. This is the reason why we restrict our numerical results to 2D computations.
\end{rev}


Following this approach, figure \ref{fig-convergence-parabolic} visualizes the convergence history of different BFGS strategies compared to a pure gradient method for problem (\ref{oc2}-\ref{oc6}).
It can clearly be seen that the BFGS methods are superior to the gradient based method.
Furthermore, we partly obtain superlinear convergence in the BFGS case.
It is yet surprising that, in this particular test case, there is hardly any difference between the number of stored gradients in the limited memory BFGS.
This changes for the pure elliptic case of (\ref{oc2}-\ref{oc6}) leaving out the time dependence yielding
\begin{equation}
\begin{aligned}
\min \hspace{0,1cm} J(\Omega):=\intdx{&(y-\bar{y})^2}+\mu\int_{\Gi}1\hspace{.5mm}ds\\
\mbox{s.t. } - \mathrm{div}(k\nabla y)&=f\quad \text{in }\Omega
\\
\hspace{20mm}y&=1\quad \text{on }\Gt
\\
\hspace{20mm}y&=0\quad \text{on }\Gb\
\\
\nd{y}&=0\quad \text{on } \Gb\cup\Gl\cup\Gr
\end{aligned}
\end{equation}
Note that the boundary conditions are changed compared to the parabolic model since these conditions would lead to a homogeneous steady state distribution of $y$.
The shape gradient for this problem can be found \cite{ItoKunisch}.
Here we observe small improvements in the convergence while enlarging the memory width for the BFGS method, which is visualized in figure \ref{fig-convergence-elliptic}.

\begin{rev}
Back in the parabolic case, we also investigate the influence of the grid on the convergence, which is depicted in figure \ref{fig-different-grids}.
Two grids are tested. A coarse one with approximately 25,000 cells and a much finer grid with about 100,000 cells.
It can be seen here that the convergence is almost grid independent for both the gradient and the BFGS method.
This also visualizes the discretization error.
\end{rev}

\begin{figure}
\centering
\subfigure[Optimal shapes\label{fig_optimized_shapes}]{\includegraphics[width=0.45\textwidth]{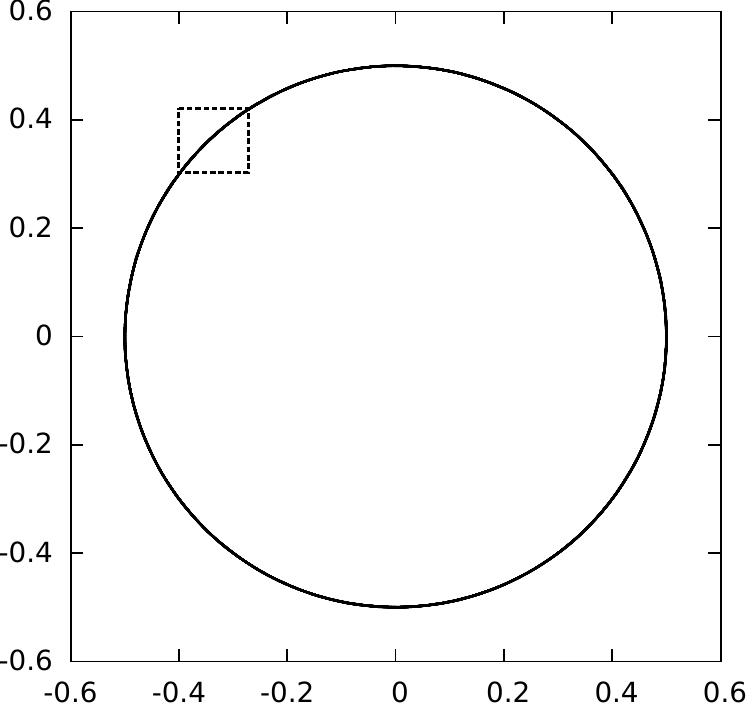}}
\subfigure[Zoom into the dashed frame\label{fig_zoom}]{\includegraphics[width=0.45\textwidth]{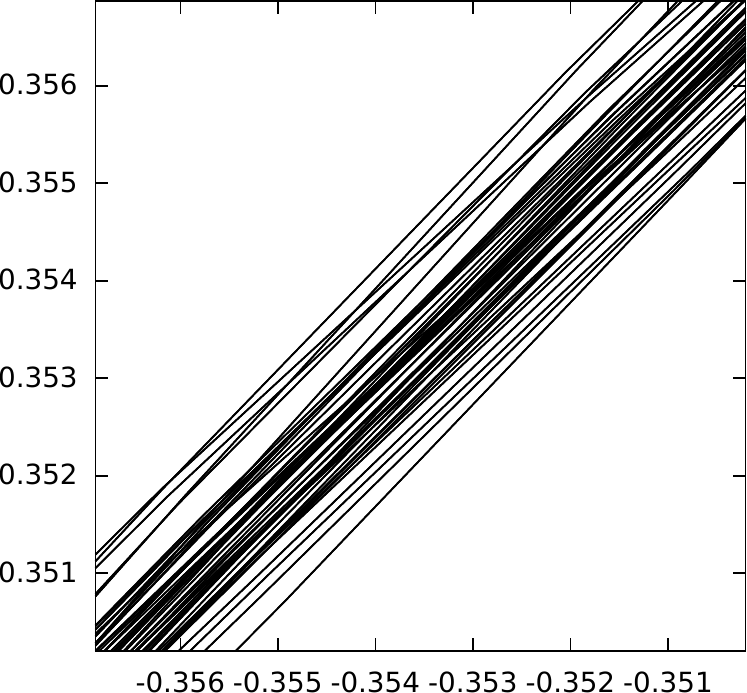}}               
\caption{100 optimized shapes with 5\% noise in measurements $\bar{y}$}
\label{fig_disturbed}
\end{figure}

\begin{rev}
In a final test run we investigate convergence under noisy measurements.
We therefore add white noise $\omega(t,x)$ to the measurements $\bar{y}(t,x)$ with an amplitude of 5\% of the maximum value of $\bar{y}$, which is $1.0$ due to the boundary conditions.
Then we perform 100 runs of the optimization algorithm in the setting described in the beginning of section \ref{numex}.
Due to the disturbed measurements we obtain slightly different optimal shapes.
In order to estimate the difference between these shapes, the maximum, point-wise distance is evaluated.
We observe that this distance is only 0.21\% of the mean diameter of all converged shapes, which is  relatively small compared to the noise added to the measurements.
In figure \ref{fig_optimized_shapes} all 100 converged shapes are visualized. From this point of view there are hardly any differences noticeable. Figure \ref{fig_zoom} shows a zoom into the region framed with dashed lines with approximately 200x magnification. Furthermore, we observed in these experiments that one can also use a regularization parameter $\mu=0$ without a noticeable difference. 
\end{rev}


\section{Conclusions}
This paper develops a novel shape gradient for structured inverse modeling in diffusive processes. The second novelty of this paper lies in the application of quasi-Newton methods in shape space. We observe very fast convergence to the level of the approximation error -- and this without any line-search. These promising results are to be extended to more practically challenging problems in a large-scale framework in subsequent papers.  

\section*{Acknowledgment}
The authors are very grateful for several important suggestions for improvement of the paper mentioned by the two anonymous referees and the editor. This work has been partly supported by the Deutsche Forschungsgemeinschaft within the Priority program SPP 1648 ``Software for Exascale Computing'' under contract number
Schu804/12-1. 
 

\begin{thebibliography}{10}

\bibitem{Absil-book-2008}
{P.-A.} Absil, R.~Mahony, and R.~Sepulchre.
\newblock {\em Optimization Algorithms on Matrix Manifolds}.
\newblock Princeton University Press, 2008.

\bibitem{BHM-2011}
{M.} Bauer, P.~Harms, and {P. W.} Michor.
\newblock Sobolev metrics on shape space of surfaces.
\newblock {\em Journal of Geometric Mechanics}, 3(4):389--438, 2011.

\bibitem{Berggren}
M.~Berggren.
\newblock A unified discrete-continuous sensitivity analysis method for shape
  optimization.
\newblock In W.~Fitzgibbon et~al., editor, {\em Applied and numerical partial
  differential equations}, volume~15 of {\em Computational {M}ethods in
  {A}pplied {S}iences}, pages 25--39. Springer, 2010.

\bibitem{ColtonKress}
D.~Colton and R.~Kress.
\newblock {\em Inverse acoustic and electromagnetic scattering theory},
  volume~93 of {\em Science and Business Media}.
\newblock Springer, 3 edition, 2012.

\bibitem{CorreaSeger}
R.~Correa and A.~Seeger.
\newblock Directional derivative of a minmax function.
\newblock {\em Nonlinear Anal.}, 9(1):13--22, 1985.

\bibitem{Delfour-Zolesio-2001}
M.~C. Delfour and J.-P. Zol\'esio.
\newblock {\em Shapes and Geometries: Analysis, Differential Calculus, and
  Optimization}.
\newblock Advances in Design and Control. SIAM Philadelphia, 2001.

\bibitem{Epp-Har-2005}
K.~Eppler and H.~Harbrecht.
\newblock A regularized newton method in electrical impedance tomography using
  shape {H}essian information.
\newblock {\em Control and Cybernetics}, 34(1):203--225, 2005.

\bibitem{GabayDxX1982b}
D.~Gabay.
\newblock Minimizing a differentiable function over a differential manifold.
\newblock {\em Journal of Optimization Theory and Applications},
  37(2):177--219, 1982.

\bibitem{GrossReusken}
S.~Gross and A.~Reusken.
\newblock {\em Numerical methods for two-phase incompressible flows}, volume~40
  of {\em Computational Mathematics}.
\newblock Springer, 2010.

\bibitem{harbrecht-2014}
H.~Harbrecht and J.~Tausch.
\newblock On shape optimization with parabolic state equation.
\newblock Technical Report Preprint 2013-23, Mathematisches Institut,
  Universit\"at Basel, 2013.

\bibitem{HaslingerMakinen}
J.~Haslinger and {R. A. E.} M\"{a}kinen.
\newblock {\em {Introduction to Shape Optimization: Theory, Approximation, and
  Computation}}.
\newblock Advances in Design and Control. SIAM Philadelphia, 2003.

\bibitem{HettlichRundell}
F.~Hettlich and W.~Rundell.
\newblock A second degree method for nonlinear inverse problems.
\newblock {\em SIAM Journal on Numerical Analysis}, 37(2):587--620, 2010.

\bibitem{ItoKunisch}
K.~Ito and K.~Kunisch.
\newblock {\em Lagrange Multiplier Approach to Variational Problems and
  Applications}, volume~15 of {\em Advances in Design and Control}.
\newblock SIAM Philadelphia, 2008.

\bibitem{McLaughlin1996}
D.~Mc{L}auglin and {L. R.} Townley.
\newblock A reassessment of the groundwater inverse problem.
\newblock {\em Water Resources Research}, 32(5):1131--1161, 1996.

\bibitem{meyer2003discrete}
M.~Meyer, M.~Desbrun, P.~Schr{\"o}der, and A.~H. Barr.
\newblock Discrete differential-geometry operators for triangulated
  2-manifolds.
\newblock In {\em Visualization and mathematics III}, pages 35--57. Springer,
  2003.

\bibitem{MM-2006}
{P. W.} Michor and D.~Mumford.
\newblock Riemannian geometries on spaces of plane curves.
\newblock {\em J. Eur. Math. Soc. (JEMS)}, 8:1--48, 2006.

\bibitem{neagel20015scalable}
A.~N\"agel, V.~Schulz, M.~Siebenborn, and G.~Wittum.
\newblock Scalable methods for structured inverse modelling in diffusive
  processes.
\newblock {\em Computing and Visualization in Science}, 2015 (submitted).

\bibitem{Nocedal-Wright}
J.~Nocedal and S.~J. Wright.
\newblock {\em Numerical optimization}.
\newblock Springer, 2000.

\bibitem{NovruziRoche}
A.~Novruzi and J.~R. Roche.
\newblock {Newton's method in shape optimisation: a three-dimensional case}.
\newblock In {\em BIT Numerical Mathematics}, volume~40, pages 102--120.
  Springer, 2000.

\bibitem{Novruzi-2002}
Arian Novruzi and Michel Pierre.
\newblock Structure of shape derivatives.
\newblock {\em Journal of Evolution Equations}, 2:365--382, 2002.

\bibitem{Paganini}
A.~Paganini.
\newblock Approximative shape gradients for interface problems.
\newblock Technical Report 2014-12, Seminar for Applied Mathematics, ETH
  Z{\"u}rich, 2014.

\bibitem{Ring-Wirth-2012}
W.~Ring and B.~Wirth.
\newblock Optimization methods on {R}iemannian manifolds and their application
  to shape space.
\newblock {\em {SIAM} Journal of Optimization}, 22:596--627, 2012.

\bibitem{VHS-shape-Riemann}
{V. H.} Schulz.
\newblock A {R}iemannian view on shape optimization.
\newblock {\em Foundations of Computational Mathematics}, 14:483--501, 2014.

\bibitem{SchulzVxX1999s}
V.~H. Schulz, A.~Bardossy, and R.~Helmig.
\newblock Conditional statistical inverse modeling in groundwater flow by
  multigrid methods.
\newblock {\em Computational Geosciences}, 3:49--68, 1999.

\bibitem{SchulzVxH1997c}
{V. H.} Schulz and G.~Wittum.
\newblock Multigrid optimization methods for stationary parameter
  identification problems in groundwater flow.
\newblock In W.~Hackbusch and G.~Wittum, editors, {\em Multigrid Methods V},
  pages 276--288. Springer, 1997.

\bibitem{SokoZol}
J.~Sokolowski and {J.-P.} Zol\'{e}sio.
\newblock {\em An introduction to shape optimization}.
\newblock Springer, 1992.

\bibitem{Sturm2013}
Kevin Sturm.
\newblock Lagrange method in shape optimization for non-linear partial
  differential equations: A material derivative free approach.
\newblock Technical Report No. 1817, WIAS Berlin, 2013.

\bibitem{Troeltzsch}
F.~Tr\"oltzsch.
\newblock {\em {Optimal Control of Partial Differential Equations: Theory,
  Methods, and Applications}}, volume 112 of {\em Applied Mathematics}.
\newblock American Mathematical Society, 2010.

\end{thebibliography}
\bibliographystyle{plain}

\end{document}